\documentclass[a4paper,12pt]{article}
\usepackage{amssymb}
\usepackage{bm}
\usepackage{amsmath,amsthm}
\usepackage{amscd}
\usepackage[arrow,matrix]{xy}
\usepackage[bookmarksnumbered,colorlinks,linkcolor=black,citecolor=black,urlcolor=black]{hyperref}
\usepackage{mathtools}
\usepackage{indentfirst} % YY to indent in the beginning of every section
\usepackage{mathrsfs}

\usepackage{tikz-cd}
\usepackage{tikz}
\usepackage{enumitem}

\newtheorem{theo}{Theorem}[section]
\newtheorem{prop}[theo]{Proposition}
\newtheorem{lem}[theo]{Lemma}
\newtheorem{cor}[theo]{Corollary}
\newtheorem{conj}[theo]{Conjecture}

\newtheorem{theon}{Theorem}
\newtheorem{corn}[theon]{Corollary}
\newtheorem*{defiu}{Definition}

\theoremstyle{definition}
\newtheorem{defi}[theo]{Definition}
\newtheorem{rem}[theo]{Remark}
\newtheorem{exa}[theo]{Example}
\newtheorem{ques}[theo]{Question}
\newtheorem{exe}[theo]{Exercise}

\newcommand{\bthe}{\begin{theo}}
\newcommand{\ble}{\begin{lem}}
\newcommand{\bpr}{\begin{prop}}
\newcommand{\bco}{\begin{cor}}
\newcommand{\bde}{\begin{defi}}
\newcommand{\ethe}{\end{theo}}
\newcommand{\ele}{\end{lem}}
\newcommand{\epr}{\end{prop}}
\newcommand{\eco}{\end{cor}}
\newcommand{\ede}{\end{defi}}
\newcommand{\brem}{\begin{rem}}
\newcommand{\erem}{\end{rem}}

\newcommand{\bexe}{\begin{exe}}
\newcommand{\eexe}{\end{exe}}

\newcommand{\bexa}{\begin{exa}}
\newcommand{\eexa}{\end{exa}}

\newcommand{\bconj}{\begin{conj}}
\newcommand{\econj}{\end{conj}}
\newcommand{\bques}{\begin{ques}}
\newcommand{\eques}{\end{ques}}

\def \nr {{\rm nr}}

\def \Br {{\rm{Br}}}

\def \si {{\sigma}}

\def \R {{\mathbb{R}}}
\def \Pic {{\rm {Pic}}}

\def \Gal {{\rm{Gal}}}

\def \A{{\mathbb A}}
\def \P{{\mathbb P}}

\def \Spec {{\rm{Spec}}}
\def \dim {{\rm{dim}}}
\def \Hom {{\rm {Hom}}}
\def \End {{\rm {End}}}
\def \Pic {{\rm {Pic}}}

\def \cris {{\rm cris}}

\def\ov{\overline}

\def \Z {{\mathbb Z}}
\def \Q {{\mathbb Q}}
\def \F {{\mathbb F}}

\def \rk {{\rm{rk}}}

\def \H {{\rm H}}
\def\G{{\mathbb G}}

\def\lra{\longrightarrow}

\def \s {{\rm s}}

\def\X{{\cal X}}

\def\inv{{\rm inv}}
\def\Kum{{\rm Kum}}

\def\NS{{\rm NS}}

\def\O{{\cal O}}

\def\si{\sigma}

\def\res{{\rm res}}

\def\ev{{\rm ev}}

\def\et{{\rm{\acute et}}}

\newcommand{\isomto}{\overset{\cong}{\longrightarrow}}

\usepackage[textheight=220mm,textwidth=150mm]{geometry}

\title{Dichotomy for the Brauer--Manin obstruction in characteristic $p$}
\author{Christopher D. Lazda and Alexei N. Skorobogatov}

\date{\today}
\begin{document} 
\maketitle

\begin{abstract} 
Let $X$ be a smooth, projective, geometrically integral variety over a global function field $k$ 
of characteristic $p$. 
We show that if the unipotent Brauer group of $X$ is zero, and the Picard scheme of $X$
 is torsion free,
then only finitely many places of $k$ can be potentially relevant to the Brauer--Manin obstruction for $X$. In the opposite case, almost all places of $k$ are potentially relevant. If $X$ is base changed from a finite field, then in both cases the finite set of exceptional places is empty.
We give an example of a supersingular K3 surface over a global function field $k$ of 
characteristic $p>2$
such that $\Br(X)$ is finite modulo $\Br(k)$, answering a question from our previous paper.
\end{abstract}

\section*{Introduction}

Let $X$ be a variety over a global field $k$. A place $v$ of $k$ is called {\em relevant}
to the Brauer--Manin obstruction on $X$ if there is a Brauer class $\mathcal{A}\in \Br(X)$
such that $\mathcal{A}(P)\neq \mathcal{A}(Q)$ for some $k_v$-points $P,Q\in X(k_v)$.
A place $v$ of $k$ is called {\em potentially relevant} if there is a finite field extension $L/k$
and a place of $L$ above $v$ that is relevant to the Brauer--Manin obstruction on $X\times_kL$. Note that in positive characteristic we allow the extension $L$ to be inseparable.

In their recent papers, Bright and Newton \cite[Theorem C]{BN} 
and  Ambrosi, Newton, and Pagano \cite[Corollary 1.1.3]{ANP} proved that if 
$X$ is a smooth, proper, geometrically integral
variety over a number field $k$, such that $\H^2(X,\O_X)\neq 0$, 
then almost every place $v$ of $k$ is potentially relevant to the Brauer--Manin obstruction. 

Now let $X$ be a smooth, projective, geometrically integral variety over a global field $k$ of 
{\em positive} characteristic. One can ask the following natural question: for which
varieties $X$ are almost all
places of $k$ potentially relevant to the Brauer--Manin obstruction? 

In this paper we give a complete answer to this question.
Let $\bar k$ be an algebraic closure of $k$. 
By work of Milne, the quotient of $\Br(X_{\bar k})[p^\infty]$ by its divisible subgroup is the group of 
$\bar k$-points of a quasi-algebraic group $G_X$ in the sense of Serre \cite{Serre60},
see \cite[\S 2]{Ber}. By \cite[Proposition 7, p.~10]{Serre60}, $G_X$ is 
an extension of a finite \'etale group $D_{X}$ by a smooth connected unipotent group $U_{X}$.

\begin{defiu} 
We call $U_X$ the unipotent Brauer group of $X$.
\end{defiu}

For example, if $X$ is a K3 surface or an abelian surface, 
then $U_X\neq 0$ if and only if $X$ is supersingular,
in which case $U_X$ is the quasi-algebraic group $\G_a$ and $\Br(X_{\bar k})\cong \bar k$. If $A$ is an abelian variety, then $U_A=0$ if and only if $A$ is ordinary or almost ordinary \cite[Remark 3.2]{GSY}.
In general, there is an explicit criterion for the non-vanishing of $U_X$ in terms of the slopes of the crystalline cohomology groups of $X_{\bar k}$, see Lemma \ref{lemma1} below.

Let ${\bf A}_k$ be the ad\`ele ring of $k$, and $X({\bf A}_k)^\Br$ the Brauer--Manin set of $X$.
The second named author asked about the role played by the unipotent Brauer group 
in the Brauer--Manin obstruction for $X$. In response to this question, Valloni has shown that
if $X$ is a supersingular K3 surface, then up to replacing $k$
by a finite inseparable extension, the Brauer--Manin set $X({\bf A}_k)^\Br$
is `very small'; in particular it is not open in $X({\bf A}_k)$, and every place of $k$ is relevant
to the Brauer--Manin obstruction \cite[Theorem 1.3]{V}. In contrast, if $X$ is a 
non-supersingular K3 surface, then $\Br(X_{\bar k})$ is divisible, 
and $\Br(X)$ is finite modulo the image of $\Br(k)$
by \cite{SZ15, Ito18, LQ24, LS24}, implying that $X({\bf A}_k)^\Br$ is open in $X({\bf A}_k)$.
In particular, only finitely many places of $k$ are relevant to the Brauer--Manin obstruction on $X$.
Note, however, that this description leaves open the question about potential relevance.

We show that Valloni's theorem reflects a basic dichotomy in the behaviour of the Brauer--Manin obstruction. Write ${\bf Pic}_{X/k}$ for the Picard scheme of $X$. Our main result is as follows. 

\begin{theon} \label{theo: main}
Let $X$ be a smooth, projective, geometrically integral variety over a global function field $k$ of characteristic $p>0$. Let $\ell$ be a prime number.
\begin{enumerate}[label=\normalfont(\roman*)]
\item \label{Ai} Let $\ell=p$. If both $U_X=0$ and ${\bf Pic}_{X/k}[p]= 0$, then only finitely many places of $k$ are potentially relevant to the $p$-primary Brauer--Manin obstruction for $X$. Otherwise, almost all places are potentially relevant to the $p$-primary Brauer--Manin obstruction for $X$.
\item \label{Aii} Let $\ell\neq p$. If ${\bf Pic}_{X/k}[\ell]=0$, then only finitely many places of $k$ are potentially relevant to the $\ell$-primary Brauer--Manin obstruction for $X$. Otherwise, almost all places are potentially relevant to the $\ell$-primary Brauer--Manin obstruction for $X$.
\end{enumerate}
\end{theon}

The hypothesis ${\bf Pic}_{X/k}[\ell]=0$ implies that the Picard scheme of $X$
has dimension~$0$. If $\ell=p$, it even implies that ${\bf Pic}_{X/k}$ is \'etale.

 The proof of Theorem \ref{theo: main}\ref{Aii} depends on relatively standard calculations in \'etale cohomology, and one can be reasonably explicit about the set of exceptional places in each case: they are simply the places of bad reduction. In particular, they are independent of $\ell\neq p$, and so we get:

\begin{corn} \label{cor: main}Let $X$ be a smooth, projective, geometrically integral variety over a global function field $k$. If $U_X=0$ and ${\bf Pic}_{X/k}$ is torsion-free, then only finitely many places of $k$ are potentially relevant to the Brauer--Manin obstruction for $X$. Otherwise, almost all places are potentially relevant to the Brauer--Manin obstruction for $X$.
\end{corn}

The proof of Theorem \ref{theo: main}\ref{Ai} is more involved, and depends on two key ingredients: the generic representability results of Bragg--Olsson for flat cohomology \cite{BO}, and, for the second statement, recent results of Krishna and Majumder \cite[Theorem 1.1]{KM} on Kato's filtration in equicharacteristic $p$. The dependence on the former in particular makes it more difficult to be explicit about the precise set of places required. Indeed, if $U_X$ and ${\bf Pic}_{X/k}[p]$ both vanish, then the set of places which are not potentially relevant comes from a smooth spreading out $\pi\colon \X\to C$ of $X$, where $C$ is a smooth curve with function field $k$, for which 
the fppf sheaf ${\bf R}^2\pi_*\mu_p$ is representable by a finite flat group scheme and ${\bf R}^3\pi_*\mu_p$ is representable by an affine group scheme. If either $U_X$ or ${\bf Pic}_{X/k}[p]$ is non-zero, then all places of $C$ are potentially relevant provided that ${\bf R}^1\pi_*\mu_p$ is representable by a finite flat group scheme, and ${\bf R}^2\pi_*\mu_p$ by an affine group scheme of finite type. The results of Bragg--Olsson guarantee the existence of such a curve $C$ in each case, but (to the best of our knowledge) give little control over how large we can take $C$ to be.

In some cases, however, $C$ can be made explicit. For example, when $X$ is constant we can take $C$ to be the unique smooth projective curve with function field $k$. This leads to the following: 

\begin{theon}[Corollary \ref{cor: const zero}, Theorem \ref{theo: const UXn0}]\label{theo: main cons}
Let $X$ be a smooth, projective, geometrically integral variety over a finite field $\F$, and let $k/\F$ be a global function field.
\begin{enumerate}[label=\normalfont(\roman*)]
\item If $U_X=0$ and ${\bf Pic}_{X/\F}$ is torsion-free, then $X({\bf A}_k)^\Br=X({\bf A}_k)$.
\item If $U_X\neq 0$, or ${\bf Pic}_{X/\F}[\ell]\neq 0$ for some prime $\ell$,  then every place of $k$ is potentially relevant to the Brauer--Manin obstruction for $X_k$.
\end{enumerate}  
\end{theon}

Another case when $C$ can be made explicit is when the geometric Brauer group
of $X$ is $p$-torsion free, that is, $\Br(X_{\bar k})[p^\infty]=0$. Then, as in the $\ell$-adic case, we can
take $C$ to consist of the places of good reduction, see Proposition \ref{a4}. 
We give examples of non-isotrivial K3 surfaces with 
$\Br(X_{\bar k})[p^\infty]=0$ (Proposition \ref{a5}), so this case is not subsumed by Theorem \ref{theo: main cons}.

One may enquire about the field extensions needed to make places relevant to the Brauer--Manin
obstruction. We give examples of supersingular K3 surfaces over global function fields of characteristic $p>2$
such that $\Br(X)/\Br_0(X)$ is finite, see Proposition \ref{5.3}, answering a question raised by the authors in their previous paper \cite{LS24}. In this case
the Brauer--Manin set $X({\bf A}_k)^\Br$ is open in $X({\bf A}_k)$, in particular, 
only finitely many primes of $k$ are relevant to the Brauer--Manin
obstruction for $X$. Thus in general one cannot drop `potential' from
`potentially relevant' in the case $U_X\neq 0$ of Corollary~B.
This also shows that in general a field extension in \cite[Theorem 1.3]{V} is necessary.
In our example,
the finiteness of $\Br(X)/\Br_0(X)$ is due to the fact that the connected component of
the group $k$-scheme ${\bf R}^2\pi_*\mu_p$, where $\pi\colon X\to \Spec\, k$ is the structure morphism,
is a {\em non-split} form of $\G_{a,k}$. Indeed, when $k$ is a global function field of
characteristic $p>2$, the non-split forms of $\G_{a,k}$ have
only finitely many $k$-points, in contrast to $\G_{a,k}$. We use
work of Bragg and Lieblich \cite{BL18}, which also shows that for any supersingular K3 surface
of Artin invariant $\si$
the connected component of ${\bf R}^2\pi_*\mu_p$ is split by the $(\si-1)$-fold Frobenius twist
(Proposition \ref{prop: BL}, see also
Corollary \ref{cococo} for a similar statement for abelian surfaces).

Here is a brief outline of this paper.
In Section \ref{1} we put together known results from the literature to give a formula
for the dimension of $U_X$ in terms of the slopes of 
the crystalline cohomology groups of $X_{\bar k}$.
In Section \ref{2} we combine the representability results of Bragg--Olsson from \cite{BO} with the properties of the evaluation filtration on the local Brauer group established by Krishna and Majumder in \cite{KM} to prove our main result, Theorem \ref{theo: main}.
In Section \ref{2.5} we consider the case of constant varieties, and prove Theorem \ref{theo: main cons}, and in Section \ref{t-free} we discuss varieties with $p$-torsion free geometric Brauer group.
Finally, Section \ref{5} is devoted to the study of ${\bf R}^2\pi_*\mu_p$
for K3 and abelian surfaces, and the example mentioned in the previous paragraph.

\section{Brauer groups in positive characteristic} \label{1}

Let $X$ be a smooth, proper, geometrically integral variety over a field $k$ of characteristic $p>0$,
with algebraic closure $\bar k$. By a theorem of Grothendieck, 
$\Br(X_{\bar k})=\H^2_\et(X_{\bar k},\G_m)$ is a torsion group \cite[Theorem 3.5.5]{CTS21}.
For primes $\ell\neq p$, the structure of 
$\Br(X_{\bar k})[\ell^\infty]$ was determined also by Grothendieck, 
see \cite[Proposition 5.2.9]{CTS21}.
We now discuss the structure of $\Br(X_{\bar k})[p^\infty]$.

The dimension of the unipotent Brauer group $U_X$, as well as the corank of the divisible subgroup of $\Br(X_{\bar k})[p^\infty]$,
can be calculated in terms of the crystalline cohomology groups of $X_{\bar k}$ as follows.
Let $W=W(\bar k)$ be the Witt ring of $\bar k$, with field of fractions $K$.
Write
$$\rho=\rk\, \NS(X_{\bar k}) \quad\text{and}\quad b_2=\dim_K\big( \H^2_\cris(X_{\bar k}/W)\otimes K\big),$$
for the Picard rank and the second Betti number, respectively. Write
$$h=\dim_K\big (\H^2(X_{\bar k},W\O_{X_{\bar k}})\otimes K\big).$$
For any $\lambda \in \Q$, we let $\big(\H^i_\cris(X_{\bar k}/W)\otimes K\big)_{[\lambda]}$ denote the slope $\lambda$ summand of $\H^i_\cris(X_{\bar k}/W)\otimes K$, and for an interval $I\subset \R$, we let  $\big(\H^i_\cris(X_{\bar k}/W)\otimes K\big)_{I}$ denote the maximal summand of $\H^i_\cris(X_{\bar k}/W)\otimes K$ with all slopes in $I$. 
Since the slope spectral sequence degenerates modulo torsion on the first page,
there is, for $i\geq 0$, an isomorphism of isocrystals
$$\H^2(X_{\bar k},W\O_{X_{\bar k}})\otimes K\cong \big(\H^i_\cris(X_{\bar k}/W)\otimes K\big)_{[0,1)},$$
see \cite[II, (3.5.3)]{Ill}. 
We thus have
$$h=\dim_K \big(\H^2_\cris(X_{\bar k}/W)\otimes K\big)_{[0,1)}.$$
If $X$ is projective, we have moreover 
$$\dim_K \big(\H^2_\cris(X_{\bar k}/W)\otimes K\big)_{[0,1)} =\dim_K \big(\H^2_\cris(X_{\bar k}/W)\otimes K\big)_{(1,2]},$$ see \cite[p.~630]{Ill}. 

\ble \label{a2}
Let $X$ be a smooth, projective, geometrically integral variety over a field $k$ of characteristic $p$.
The divisible subgroup of $\Br(X_{\bar k})[p^\infty]$ is isomorphic to $(\Q_p/\Z_p)^{b_2-2h-\rho}$.
\ele
\begin{proof} 
This follows from \cite[II, Proposition 5.12]{Ill}.
\end{proof}

\ble \label{a3}
Let $X$ be a smooth, projective, geometrically integral variety over a field $k$ of characteristic $p$
such that $U_{X}=0$. 
Let $\bar k\subset K$ be an extension of algebraically closed fields. Then the natural map
$\Br(X_{\bar k})\to \Br(X_K)$ is an isomorphism.
\ele
\begin{proof} This map is injective \cite[Proposition 5.2.2]{CTS21}, and induces an isomorphism of
the prime-to-$p$ torsion subgroups \cite[Proposition 5.2.3]{CTS21}. 

The invariants $b_2,h$ and $\rho$ do not change in passing from $\bar k$ to $K$.
Thus Lemma \ref{a2} shows that the restriction of 
${\Br}(X_{\bar k})[p^\infty]\to \Br (X_K)[p^\infty]$
to the divisible subgroups is an injective map of the form
\[ (\Q_p/\Z_p)^{b_2-2h-\rho}\to (\Q_p/\Z_p)^{b_2-2h-\rho},  \]
which is therefore an isomorphism.

Since $U_{X}=0$, the quotient of $\Br(X_{\bar k})[p^\infty]$ by its divisible subgroup is the group of 
$\bar k$-points of the finite \'etale group $D_{X}$.
The pullback map induces an isomorphism $D_{X}(\bar k)\to D_{X}(K)$, and the lemma follows. 
\end{proof}

\medskip

Next, let us define
$$h^{0,i}=\dim_k\, \H^i(X,\O_X), \quad
m^{0,i}=\sum_{\lambda\in[0,1)}(1-\lambda) \dim_K (\H^2(X_{\bar k},W\O_{X_{\bar k}})\otimes K)_{[\lambda]}.$$
The numbers $m^{0,i}$ were introduced by Ekedahl and Crew. By the above, 
the number $m^{0,i}$ only depends on the slopes of $\H^i_\cris(X_{\bar k}/W)$.

\ble
Let $X$ be a smooth, projective, geometrically integral variety over a field $k$ of characteristic $p$.
Then the Picard scheme ${\bf Pic}_{X/k}$ is smooth if and only if $h^{0,1}-m^{0,1}=0$.
\ele
\begin{proof} It is well known (see, e.g., \cite[Theorem 5.1.1]{CTS21})
that in this situation the group scheme ${\bf Pic}_{X/k}$ 
exists and represents the relative Picard functor;
the connected component ${\bf Pic}^\circ_{X/k}$ is projective,
with tangent space at the origin $\H^1(X,\O_X)$. To check smoothness, we may base change to $\bar k$, and thus assume that $k$ is algebraically closed. 

In this case, $A:={\bf Pic}^\circ_{X/k, {\rm red}}$ is an abelian variety over $k$, 
and so ${\bf Pic}^\circ_{X/k}$ is smooth if and only if $h^{0,1}=\dim\, A$.
Let $A^\vee$ be the dual abelian variety of $A$. Thus $A^\vee$ is the Albanese variety of $X$, so 
$\H^1_\cris(X_{\bar k}/W)\cong \H^1_\cris(A^\vee_{\bar k}/W)$ is isomorphic to
the contravariant Dieudonn\'e module of $A^\vee$, which is the same as the
covariant Dieudonn\'e module of $A$. The dimension of the $p$-divisible group of $A$
is calculated in terms of slopes
of the covariant Dieudonn\'e module of $A$ by the same formula that defines $m^{0,1}$,
see \cite[Proof of Corollary III.3.4]{AM} or \cite[Lemma 3.6]{GSY}.
\end{proof}

\ble \label{lemma1}
Let $X$ be a smooth, projective, geometrically integral variety over a field $k$ of characteristic $p$,
such that ${\bf Pic}_{X/k}$ is smooth. Then the following properties hold:
\begin{enumerate}[nosep, label=\normalfont(\roman*)]
\item the formal Brauer group $\widehat\Br(X)$ is representable by a formal group scheme;
\item if $\H^3_\cris(X/W)$ is torsion-free, then $\widehat\Br(X)$ is smooth;
\item if $\widehat\Br(X)$ is smooth, then $\dim(U_X)=h^{0,2}-m^{0,2}$.
\end{enumerate}

\ele
\begin{proof}
\begin{enumerate}[nosep, label=(\roman*)]
\item This is a result of Artin and Mazur \cite[Corollary II.4.1]{AM}, see 
also \cite[v1, Proposition 10.11]{BO}.
\item Smoothness is fpqc local on the base, so we can check the smoothness of $\widehat\Br(X)$ 
over $\bar k$, where the result follows from \cite[Theorem 1.2]{G}.
\item By definition, this is calculated after base changing to $\bar k$, where the result follows from 
\cite[Proposition 3.7]{GSY}. \qedhere
\end{enumerate}
\end{proof}

Now, let $\pi\colon X\to \Spec\,k$ be the structure morphism. Results of Bragg--Olsson show that if $X$ is projective, then each fppf sheaf ${\bf R}^q\pi_{*}\mu_{p^n}$ is representable by an affine group scheme of finite type over $k$, see \cite[Corollary 1.6]{BO}. We then have the following relationship with the unipotent Brauer group $U_X$:

\ble \label{lem: UX} 
Let $\pi\colon X\to \Spec\,k$ be a smooth, projective, geometrically integral variety over a field $k$ of characteristic $p$.
Then the following are equivalent:
\begin{enumerate}[nosep, label=\normalfont(\roman*)]
\item \label{UXi} $U_X= 0$;
\item \label{UXii} for all $n\geq 1$, ${\bf R}^2\pi_{*}\mu_{p^n}$ is a finite (flat) group scheme over $k$;
\item \label{UXiii} for some $n\geq 1$, ${\bf R}^2\pi_{*}\mu_{p^n}$ is a finite (flat) group scheme over $k$.
\end{enumerate}
\ele
\begin{proof} We may base change to $\bar k$, and thus assume that $k$ is algebraically closed. In this case ${\bf R}^2\pi_{*}\mu_{p^n}$ is isomorphic to a product of a unipotent group 
scheme ${\bf U}_n$ and a group scheme of multiplicative type $H$, by \cite[Theorem 16.13]{Mil17}. 
Since ${\bf R}^2\pi_{*}\mu_{p^n}$ is $p^n$-torsion, it cannot have $\G_m$ as a direct factor, hence $H$ is a finite group scheme of multiplicative type, annihilated by $p^n$.
Thus $H(k)$ is trivial, so we have 
\[ \H^2(X,\mu_{p^n})= \H^0(k,{\bf R}^2\pi_*\mu_{p^n}) = {\bf U}_n(k).  \]
If $U_X=0$, we see that $G_X=D_X$ is a finite \'etale group scheme, thus
\[ \Br(X)[p^\infty] \cong (\Q_p/\Z_p)^{b_2-2h-\rho} \oplus D_X(k)   \]
and so $\Br(X)[p^n]$ is finite. We then deduce that $\H^2(X,\mu_{p^n})= {\bf U}_n(k)$ is finite from the Kummer exact sequence. Since $k$ is algebraically closed, we conclude that
${\bf U}_n$ is a finite unipotent group scheme. Hence \ref{UXi} implies \ref{UXii}.

It is clear that \ref{UXii} implies \ref{UXiii}. Finally, if $\dim\, {\bf R}^2\pi_{*}\mu_{p^n}=0$ for some $n\geq 1$, then $\H^0({\bar k},{\bf R}^2\pi_*\mu_{p^n})$ is finite, which forces $\Br(X_{\bar k})[p^n]$ to be finite, and hence $U_X=0$. Thus \ref{UXiii} implies \ref{UXi}.
\end{proof}

\section{Main results} \label{2}

Let $k$ be a global function field of characteristic $p>0$, and ${\bf A}_k$ its ad\`ele ring. Let 
$\pi\colon X\to \Spec\,k$ be a smooth, projective, geometrically integral variety over $k$, which we assume to have a rational point. We fix a place $v$ of $k$, and let $k_v$ denote the completion of $k$ at $v$, $\O_v$ its ring of integers, and $\F_v$ its residue field. We denote the base change of $X$ to $k_v$ by $X_{k_v}$. 
Since $X(k)\neq\emptyset$, the pullback maps $\Br(k)\to \Br(X)$ and $\Br(k_v)\to \Br(X_{k_v})$ 
are split injective.

We assume that $X$ has good reduction at $v$, in the sense that there exists a smooth and proper morphism of schemes 
$\pi_v \colon \mathcal{X} \to \Spec\,\O_v$ whose generic fibre is $X_{k_v}$. For each prime number $\ell$ (including $\ell=p$), and any $n\geq 1$, we consider the sheaves ${\bf R}^q\pi_{v*}\mu_{\ell^n}$ on the (large) fppf site of $\O_v$ (and similarly for ${\bf R}^q\pi_{*}\mu_{\ell^n}$). If $\ell\neq p$, then standard theorems in \'etale cohomology imply that these are representable by finite \'etale group schemes over $\O_v$. If $\ell=p$, then the results of Bragg--Olsson \cite{BO} imply that each ${\bf R}^q\pi_{*}\mu_{p^n}$ is representable by an affine group scheme of finite type over $k$.
If $q=1$, then the structure theory of the Picard scheme shows that ${\bf R}^1\pi_{*}\mu_{p^n}={\bf Pic}_{X/k}[p^n]$ is a finite flat group scheme over $k$. We do not know if the ${\bf R}^q\pi_{v*}\mu_{p^n}$ are necessarily representable by group schemes over $\O_v$ in general.

Write $\widetilde\H^2(X,\mu_{\ell^n})$ for the cokernel of the natural map
$\H^2(k,\mu_{\ell^n})\to \H^2(X,\mu_{\ell^n})$, and use a similar notation for $X_{k_v}$.
Since $X(k)\neq \emptyset$, the Kummer exact sequences for $
\X$, $X_{k_v}$, and $X$ give rise to the following
commutative diagram with exact rows
\begin{equation} \label{11}
\begin{split}
\xymatrix@C=22pt{
0  \ar[r] & \Pic(\X)/{\ell^n} \ar[r] \ar[d]^{\cong} &  \H^2(\X,\mu_{\ell^n}) 
\ar[r]\ar@{^{(}->}[d] & \Br(\X)[\ell^n] \ar[r]\ar@{^{(}->}[d] & 0 \\
 0\ar[r] & \Pic(X_{k_v})/\ell^n \ar[r] & \widetilde\H^2(X_{k_v},\mu_{\ell^n}) \ar[r] & 
\frac{\Br(X_{k_v})[\ell^n]}{\Br({k_v})[\ell^n]}\ar[r] & 0\\ 
0\ar[r] & \Pic(X)/\ell^n \ar[r]\ar[u] & \widetilde\H^2(X,\mu_{\ell^n}) \ar[r]^{\ \si}\ar[u] & 
\frac{\Br(X)[\ell^n]}{\Br(k)[\ell^n]}\ar[r]\ar[u] & 0  
  } 
\end{split}
\end{equation}
Here we used that $\Br(\O_v)=0$ and $\Pic(\O_v)=0$, showing that $\H^2(\O_v,\mu_{\ell^n})=0$. Since $X(k)\neq \emptyset$, the Leray spectral sequence
$$\H^p(k,{\bf R}^q\pi_*\mu_{\ell^n})\Rightarrow\H^{p+q}(X,\mu_{\ell^n})$$gives rise to an 
exact sequence
\begin{equation} \label{eqn exact} 0 \to \H^1(k,{\bf R}^1\pi_*\mu_{\ell^n}) \to \widetilde{\H}^2(X,\mu_{\ell^n}) \to \H^0(k,{\bf R}^2\pi_*\mu_{\ell^n}) \overset{d}{\to} \H^2(k,{\bf R}^1\pi_*\mu_{\ell^n}).
\end{equation}
%, \\
%0 &\to \H^1(k_v,{\bf R}^1\pi_*\mu_{\ell^n}) \to \widetilde{\H}^2(X_{k_v},\mu_{\ell^n}) \to \H^0(k_v,{\bf R}^2\pi_*\mu_{\ell^n}) \overset{d}{\to} \H^2(k_v,{\bf R}^1\pi_*\mu_{\ell^n}), \nonumber \\
% 0 &\to \H^1(\O_v,{\bf R}^1\pi_{v*}\mu_{\ell^n}) \to \H^2(\X,\mu_{\ell^n}) \to \H^0(\O_v,{\bf R}^2\pi_{v*}\mu_{\ell^n})\overset{d}{\to} \H^2(\O_v,{\bf R}^1\pi_{v*}\mu_{\ell^n}), \nonumber 
%\end{align}
%If $\ell\neq p$, then $\H^2(\O_v,{\bf R}^1\pi_{v*}\mu_{\ell^n})=0$ by \cite[III, Lemma 1.1, Remark 1.2]{Mil86}.
There are similar exact sequences for $X_{k_v}$ and $\X$.

\bpr \label{cor: UX0} 
Assume that ${\bf R}^1\pi_{v*}\mu_{\ell}=0$, that ${\bf R}^2\pi_{v*}\mu_\ell$ is a finite flat group scheme over $\O_v$, and that ${\bf R}^3\pi_{v*}\mu_\ell$ is an affine group scheme over $\O_v$. Then for any finite extension  $L_w/k_v$, with ring of integers $\O_w$,  
and any $n\geq 1$, the pullback maps induce an isomorphism
\begin{equation} \label{pullback iso} \Br({L_w})[\ell^n] \oplus \Br(\X_{\O_w})[\ell^n]\isomto \Br(X_{L_w})[\ell^n].\end{equation}
In particular, for any such $L_w$, and any $\mathcal{A} \in \Br(X_{L_w})[\ell^\infty]$, the evaluation map
\[ {\rm ev}_{\mathcal{A}} \colon X(L_w)\to \Q_\ell/\Z_\ell \]
is constant.
\epr

\begin{proof}We first show by induction that, for all $n\geq 1$, ${\bf R}^1\pi_{v*}\mu_{\ell^n}=0$ and ${\bf R}^2\pi_{v*}\mu_{\ell^n}$ is a finite (but not necessarily flat) group scheme over $\O_v$. The first is clear, and for the second, we have an exact sequence
\[ 0 \to {\bf R}^2\pi_{v*}\mu_{\ell}\to {\bf R}^2\pi_{v*}\mu_{\ell^{n+1}}\to {\bf R}^2\pi_{v*}\mu_{\ell^n}\to {\bf R}^3\pi_{v*}\mu_{\ell}. \]
Now, ${\bf R}^3\pi_{v*}\mu_{\ell}$ is an affine group scheme over $\O_v$, whilst 
${\bf R}^2\pi_{v*}\mu_{\ell^n}$ is a finite group scheme over $\O_v$ by the induction hypothesis.
The kernel ${\bf K}_n$ of the map ${\bf R}^2\pi_{v*}\mu_{\ell^n}\to {\bf R}^3\pi_{v*}\mu_{\ell}$
is a closed subgroup of ${\bf R}^2\pi_{v*}\mu_{\ell^n}$, and hence a finite group scheme over $\O_v$. The exact sequence
\[ 0\to {\bf R}^2\pi_{v*}\mu_{\ell}\to {\bf R}^2\pi_{v*}\mu_{\ell^{n+1}}\to{\bf K}_n \to 0. \]
shows that ${\bf R}^2\pi_{v*}\mu_{\ell^{n+1}}$ is a ${\bf R}^2\pi_{v*}\mu_{\ell}$-torsor over ${\bf K}_n$. It is therefore representable by a scheme which is finite and flat over $\mathbf{K}_n$, and hence finite over $\O_v$.

Since ${\bf R}^1\pi_{v*}\mu_{\ell^n}=0$, the analogues of \eqref{eqn exact} 
for $\X_{\O_w}$ and $X_{L_w}$ give isomorphisms
\[ \H^2(\X_{\O_w},\mu_{\ell^n})\cong\H^0(\O_w,{\bf R}^2\pi_{v*}\mu_{\ell^{n}}),\quad\widetilde{\H}^2(X_{L_w},\mu_{\ell^n})\cong\H^0(L_w,{\bf R}^2\pi_{v*}\mu_{\ell^{n}}).\]
The group scheme ${\bf R}^2\pi_{v*}\mu_{\ell^{n}}$ is finite, so the natural map
\[ \H^0(\O_w,{\bf R}^2\pi_{v*}\mu_{\ell^n})\to \H^0(L_w,{\bf R}^2\pi_{v*}\mu_{\ell^n})\]
is bijective. The isomorphism \eqref{pullback iso} now follows from \eqref{11}.
Using that $\Br(\O_w)=0$, we deduce the claimed constancy of the evaluation maps.
\end{proof}

To obtain results in the opposite direction, we use the following proposition.
An analogous statement in the mixed characteristic case is \cite[Theorem A (1), (2)]{BN}.

\bpr[Krishna--Majumder] \label{local ev}
Let $\pi_v\colon \X\to\Spec(\O_v)$ be a smooth and proper morphism of schemes
with geometrically integral generic fibre $X_{k_v}:=\X_{k_v}$.
Let ${\rm Ev}_{-1}\Br(X_{k_v})$ denote the set of $\mathcal{A}\in\Br(X_{k_v})$ such that,
for any finite field extension $L_w/k_v$,
the evaluation map $\ev_{\mathcal{A}}\colon X(L_w)\to\Br(L_w)$ is constant.
Then ${\rm Ev}_{-1}\Br(X_{k_v})$ is the subgroup of $\Br(X_{k_v})$
generated by the images of $\Br(\X)$ and $\Br(k_v)$.
\epr
\begin{proof} 
The subgroup ${\rm Ev}_{-1}$ is part of an increasing filtration ${\rm Ev}_m\Br(X_{k_v})$ on
$\Br(X_{k_v})$, defined for $m\geq -2$, see \cite[(1.2)]{KM}. Since $\X$ is smooth and proper, it follows from \cite[Theorem 9.3]{KM} that this definition is equivalent to the Bright--Newton evaluation filtration, and hence to the definition of ${\rm Ev}_{-1}$ given above.

Since $\X_{\F_v}$ is geometrically integral, $\H^1(\F_v,\Q/\Z)$ injects into 
$\H^1(\X_{\F_v},\Q/\Z)$, and it thus follows from \cite[Theorem 1.1]{KM} that there is a residue map
$\partial \colon  {\rm Ev}_{-1}\Br(X_{k_v}) \to \H^1(\F_v,\Q/\Z)$
such that
$$\mathcal{A}\in {\rm Ev}_{-2}\Br(X_{k_v})\iff\partial(\mathcal{A})=0\iff \mathcal{A} \in 
\Br(\X)\subset\Br(X_{k_v}).$$
Since ${\Br}(k_v)\subset {\rm Ev}_{-1}\Br(X_{k_v})$ is isomorphic to $\H^1(\F_v,\Q/\Z)$ via $\partial$, the proposition follows.
\end{proof}

\brem The results of \cite{KM} are essential for the case $\ell=p$, but for the elements
of order not divisible by $p$ we can give a short direct proof. Indeed, letting $\ell\neq p$, it is clear that any class in $\Br(\X)[\ell^\infty]+\Br(k_v)[\ell^\infty]$ has constant evaluation maps after any extension. For the converse, we consider the residue map
\[ {\rm res} \colon {\rm Br}(X_{k_v})[\ell^\infty]\to \H^1(\X_{\F_v},\Q_\ell/\Z_\ell) \]
as in \cite[\S3.7]{CTS21}. The subgroup of ${\rm Br}(X_{k_v})[\ell^\infty]$ generated by $\Br(\X)[\ell^\infty]$ and $\Br(k_v)[\ell^\infty]$ is then precisely ${\rm res}^{-1}\big(\H^1(\F_v,\Q_\ell/\Z_\ell)\big)$. 

Suppose that $\mathcal{A}\in \Br(X_{k_v})[\ell^\infty]$ but $\res(\mathcal{A})\notin\H^1(\F_v,\Q_\ell/\Z_\ell)$, 
equivalently, the image of $\res(\mathcal{A})$
in $\H^1(\X_{\ov\F_v},\Q_\ell/\Z_\ell)$ has order $\ell^n$ for some positive integer $n$.
Replacing $k_v$ by a finite unramified extension we arrange that 
$\res(\mathcal{A})$ is an element of order $\ell^n$ in $\H^1(\X_{\F_v},\Z/\ell^n)\subset
\H^1(\X_{\F_v},\Q_\ell/\Z_\ell)$.
Let $f\colon Y\to \X_{\F_v}$ be a torsor with structure group $\Z/\ell^n$
whose class in $\H^1(\X_{\F_v},\Z/\ell^n)$ is $\res(\mathcal{A})$. 
By construction, $Y$ is geometrically integral. 
If $x\in \X(\F_v)$ is such that $f^{-1}(x)$ is a trivial torsor for $\Z/\ell^n$, then we have
$|f^{-1}(x)(\F_v)|=\ell^n$, otherwise $|f^{-1}(x)(\F_v)|=0$.
Thus replacing $k_v$ by a finite unramified extension,
and applying the Lang--Weil estimates, we ensure that there are points $P,Q\in \X(\F_v)$
such that $f^{-1}(P)$ is a trivial torsor whereas $f^{-1}(Q)$ is a non-trivial torsor for $\Z/\ell^n$.
Now Hensel's lemma produces $k_v$-points $\tilde P, \tilde Q\in X(k_v)$ such that 
$\mathcal{A}(\tilde P)\neq\mathcal{A}(\tilde Q)$.
\erem 

For a prime number $\ell$ we consider the natural map 
$$\eta_v \colon \H^1(k,{\bf R}^1\pi_*\mu_{\ell})\to\H^1(k_v,{\bf R}^1\pi_*\mu_\ell)/{\rm im}(\H^1(\O_v,{\bf R}^1\pi_{v*}\mu_{\ell})).$$

\ble \label{g4} Assume that ${\bf R}^1\pi_{v*}\mu_{\ell}$ is a finite flat group scheme over $\O_v$ of positive rank. Then, up to replacing $k$ by a finite extension, there exists $a\in \H^1(k,{\bf R}^1\pi_*\mu_{\ell})$ such that $\eta_v(a)\neq 0$. If $\ell\neq p$, then this extension may be taken to be separable.
\ele

\begin{proof} Throughout the proof we use the injectivity of the natural map 
$$\H^1(\O_v,G)\to \H^1(k_v,G),$$ 
where $G$ is a finite flat group scheme over $\O_v$ \cite[III, Lemma 1.1, Remark 1.2]{Mil86}.

Let $\ell\neq p$. After taking a finite separable extension of $k$ we may assume that $\mu_\ell\cong \Z/\ell$, and that ${\bf R}^1\pi_{*}\mu_{\ell}$ is a trivial Galois module, thus containing $\mu_\ell$ as a direct factor. The pullback functor from finite \'etale group schemes over $\O_v$ to those over $k_v$ is fully faithful, so this direct factor extends to a direct factor $\mu_\ell$ of ${\bf R}^1\pi_{v*}\mu_{\ell}$. It remains to note that the natural map
$k^\times/(k^\times)^\ell \to k_v^\times/\O_v^\times (k_v^\times)^\ell$
is non-zero.

Now let $\ell=p$. Over $\bar{k}$, any $p$-torsion finite flat group scheme decomposes into a direct product of \'etale, multiplicative, and local-local group schemes. Thus, after passing to a finite field extension of $k$, we can assume that ${\bf R}^1\pi_{*}\mu_{p}$ contains a direct factor $G_0$ which is either isomorphic to $\Z/p$ or $\mu_p$, or is connected and contains $\alpha_p$ as a closed subgroup. Let $G_{0,v}$ denote the closure of $G_{0,k_v}$ inside ${\bf R}^1\pi_{v*}\mu_{\ell}$; this is again a finite flat group scheme over $\O_v$. If $G_0$ is either $\Z/p$ or $\mu_p$, we set $G_1=G_0$ and $G_{1,v}=G_{0,v}$. If $G_0$ is connected and contains a copy of $\alpha_p$, we let $G_1$ denote this copy of $\alpha_p$ and $G_{1,v}$ the closure of $G_{1,k_v}$ inside $G_{0,v}$. In particular, the maps 
$$\H^1(k_v,G_1) \to  \H^1(k_v,{\bf R}^1\pi_{*}\mu_{p}),\quad\quad\H^1(k,G_1)\to  
\H^1(k,{\bf R}^1\pi_{*}\mu_{p})$$ 
are injective. This is immediate when $G_1$ is $\Z/p$ or $\mu_p$, because then $G_1$
is a direct factor; in the local-local case, we use $\H^0(k,G_0/\alpha_p)=0$.
Thus we have a commutative diagram with exact rows
\begin{equation} \label{G_1}
\begin{split}
\xymatrix@C=22pt{
 0 \ar[r] & \H^1(\O_v,G_{1,v}) \ar[r] \ar@{^{(}->}[d] &  \H^1(\O_v,{\bf R}^1\pi_{v*}\mu_{\ell})  
\ar[r]\ar@{^{(}->}[d] & \H^1(\O_v,{\bf R}^1\pi_{v*}\mu_{\ell}/G_{1,v}) \ar@{^{(}->}[d] \\
 0 \ar[r] & \H^1(k_v,G_1) \ar[r] &  \H^1(k_v,{\bf R}^1\pi_{*}\mu_{\ell})  \ar[r]& \H^1(k_v,{\bf R}^1\pi_{*}\mu_{\ell}/G_1) \\
 0 \ar[r] & \H^1(k,G_1)\ar[u] \ar[r]&\H^1(k,{\bf R}^1\pi_{*}\mu_{\ell})\ar[u] &  & 
  } 
\end{split}
\end{equation}
In particular, we see that the map 
\[ \H^1(k_v,G_1)/\H^1(\O_v,G_{1,v}) \to \H^1(k_v,{\bf R}^1\pi_{*}\mu_{\ell})/\H^1(\O_v,{\bf R}^1\pi_{v*}\mu_{\ell})\]
is injective. It therefore suffices to show that the map
\[  \H^1(k,G_1) \to \H^1(k_v,G_1)/\H^1(\O_v,G_{1,v}) \]
is non-zero, with $G_1\in \{\Z/p,\mu_p,\alpha_p\}$. Letting $\wp(x)=x^p-x$, we 
calculate $\H^1(k,G_1)$ in each of these cases:
\[ \H^1(k,\Z/p) = k/\wp(k),\quad \H^1(k,\mu_p)=k^\times/(k^\times)^p,\quad  \H^1(k,\alpha_p)=k/k^p.\]
There are similar descriptions with $k_v$ in place of $k$.

Note that the group $\H^1(k_v,G_1)$ has a natural locally compact topology (arising from the discrete valuation topology on $k_v$), with respect to which the subspace $\H^1(\O_v,G_{1,v})$ is compact.
The image of $\H^1(k,G_1)\to \H^1(k_v,G_1)$ is seen to be dense by the above explicit descriptions, so it is enough to show that $\H^1(\O_v,G_{1,v})\neq \H^1(k_v,G_{1})$.

If $G_1=\Z/p$ or $\alpha_p$, then this simply follows from the fact that $\H^1(k_v,G_{1})$ is non-compact. If $G_1=\mu_p$, then the Cartier dual $G_{1,v}^\vee$ has generic fibre $\Z/p$, and the generating section $1\in G_1(k_v)=\Z/p$ extends uniquely to an element of $G_{1,v}^\vee(\O_v)$, and hence to a morphism $\Z/p\to G_{1,v}^\vee$. Dualising gives a morphism $G_{1,v}\to \mu_p$ extending the isomorphism $G_1\cong \mu_p$ on generic fibres. Hence, by functoriality of cohomology, we can reduce to the case $G_{1,v}=\mu_p$. This can then be handled by direct computation. 
\end{proof}

For a finite field extension $L/k$ and a place $w$ of $L$ above $v$, we can also consider the natural map
$$\xi_{L,w}\colon \H^0(L,{\bf R}^2\pi_*\mu_{p})\to\H^0(L_w,{\bf R}^2\pi_*\mu_{p})/\H^0(\O_w,{\bf R}^2\pi_{v*}\mu_{p}),$$
where $\O_w$ is the ring of integers of $L_w$.

\ble \label{g3}
Assume that $U_X\neq 0$, and that ${\bf R}^2\pi_{v*}\mu_{p}$ is an affine group scheme of finite type over $\O_v$. Up to replacing $k$ by a finite field extension, there exists $a\in \H^0(k,{\bf R}^2\pi_{*}\mu_{p})$ such that for any finite field extension $L/k$, and any prime $w$ of $L$ above $v$, we have $\xi_{L,w}(a)\neq 0$.
\ele
\begin{proof}
It follows from Lemma \ref{lem: UX} that  
${\bf R}^2\pi_{*}\mu_{p}$ has positive dimension, moreover, it was shown during the proof that $\bar{k}$, ${\bf R}^2\pi_{*}\mu_{p}$ is a product of a finite group scheme and a unipotent group scheme. Thus, after a finite extension, using \cite[Theorem 16.13]{Mil17}, we can ensure that ${\bf R}^2\pi_*\mu_{p}$ contains
a smooth connected unipotent group scheme of positive dimension.
Passing to yet another finite extension, we can ensure that it is split and so
contains a copy of $\G_{a,k}$ \cite[Corollary 14.55]{Mil17}. 

Let $G_v$ be the scheme theoretic closure of $\G_{a,k_v}\subset ({\bf R}^2\pi_{*}\mu_{p})_{k_v}$
inside ${\bf R}^2\pi_{v*}\mu_p$.
This is an affine group scheme of finite type over $\O_v$. Denote by $f_v$ the natural map
$G_v(\O_v)\to G_v(k_v)=\G_{a}(k_v)=k_v$. For any pair $(L,w)$ as in the statement of the lemma, the map $f_v$ naturally extends to a map $f_w\colon G_v(\O_w)\to L_w$.

Since $G_v$ is of finite type, we may take a closed embedding $G_v\hookrightarrow \A^n_{\O_v}$.
Then the isomorphism $G_{v,{k_v}}\isomto \G_{a,{k_v}}$ extends to a map $g\colon \A^n_{k_v}\to \A^1_{k_v}$. If we let $c_v\in k_v$ denote the coefficient of largest absolute value in the polynomial defining $g$, then for every
finite field extension $L_w/k_v$ we have 
$f_w(G_v(\O_w))\subset g(\O_w^n) \subset c_v\O_w$.
If $a\in \G_a(k)=k$ is such that $a/c_v\notin \O_v$ (which exists by weak approximation), then the image of $a$
in $\G_a(k_v)=k_v$ is not in $G_v(\O_v)$,
so that $\xi_{k,v}(a)\neq 0$. In this case we have $a/c_v\notin\O_w$ for any $(L,w)$, so
the image of $a$ in $\G_a(L)=L$ maps to an element of
$\G_a(L_w)$ which is not in $G_v(\O_w)$, and $\xi_{L,w}(a)\neq 0$.
\end{proof}

\bthe \label{pot rel local} Assume that ${\bf R}^1\pi_{v*}\mu_\ell$ is a finite flat group scheme over $\O_v$, and ${\bf R}^2\pi_{v*}\mu_\ell$ is an affine group scheme of finite type over $\O_v$.
\begin{enumerate}[nosep, label=\normalfont(\roman*)]
 \item If ${\bf R}^1\pi_{v*}\mu_\ell\neq 0$, then there exists a finite extension $L/k$, a place $w$ of $L$ dividing $v$, and a class $\mathcal{A}\in \Br(X_L)[\ell]$ such that the evaluation map
\[ {\rm ev}_{\mathcal{A}}\colon X(L_w)\to \Br(L_w) \]
is non-constant.
\item If $\ell=p$, and $U_X\neq 0$, then the same conclusion holds.
\end{enumerate}
\ethe

\begin{proof} By Proposition \ref{local ev} it is enough to show that, up to possibly replacing $k$ by a finite extension, the map
$$ \Br (X)[\ell]\to \Br(X_{k_v})[\ell]/\big(\Br({k_v})[\ell]\oplus\Br(\X)[\ell]\big)$$
is non-zero. We consider diagram \eqref{11} and the exact sequence \eqref{eqn exact},
together with its analogues for $X_{k_v}$ and $\X$.

First suppose that $\ell=p$ and $U_X\neq 0$. By Lemma \ref{g3} we can find an element $x\in \H^0(k,{\bf R}^2\pi_*\mu_p)$
such that the image of $x$ in
$\H^0(L_w, {\bf R}^2\pi_*\mu_p)$ is not contained in
$\H^0(\O_w, {\bf R}^2\pi_{v*}\mu_p)$, for any finite field extension $L/k$ and any place $w$ of
$L$ above $v$. By \cite[Theorem 1.1]{Bha12}, any class in $\H^2(k, {\bf R}^1\pi_*\mu_p)$
is killed by a finite field extension of $k$. Thus we can choose $L$ such that $d(x)$ goes to
zero in $\H^2(L, {\bf R}^1\pi_*\mu_p)$, where $d$ is a map from (\ref{eqn exact}).
Then the image of $x$ in $\H^0(L,{\bf R}^2\pi_*\mu_{p^n})$ lifts to an element
$y\in\widetilde\H^2(X_L,\mu_p)$. 
We now rename $L$ by $k$ and $L_w$ by $k_v$.

Let $\mathcal A\in\Br(X)[p]$ be a lifting of $\si(y)$.
An inspection of diagram (\ref{11}) shows that the image of  
$\mathcal A$ in $\Br(X_{k_v})[p]$ is not contained in $\Br(\X_{\O_v})[p]\oplus\Br(k_v)[p]$, 
otherwise the image of $x$ in $\H^0(k_v, {\bf R}^2\pi_*\mu_p)$ would be in
$\H^0(\O_v, {\bf R}^2\pi_{v*}\mu_p)$, contrary to the choice of $x$.

Next, suppose that $\ell$ is arbitrary and ${\bf R}^1\pi_{v*}\mu_\ell\neq 0$. There is a commutative diagram with exact rows
\begin{equation} \label{R^1}
\begin{split}
\xymatrix@C=22pt{
0  \ar[r] & \H^1(\O_v,{\bf R}^1\pi_{v*}\mu_\ell) \ar[r] \ar@{^{(}->}[d] & 
 \H^2(\X_{\O_v},\mu_\ell)
\ar[r]\ar[d] & \H^0(\O_v,{\bf R}^2\pi_*\mu_\ell) \ar@{^{(}->}[d] \\
0  \ar[r] & \H^1(k_v,{\bf R}^1\pi_*\mu_\ell) \ar[r] &  \widetilde\H^2(X_{k_v},\mu_\ell)  \ar[r]&
 \H^0(k_v,{\bf R}^2\pi_*\mu_\ell) \\
0\ar[r] & \H^1(k,{\bf R}^1\pi_*\mu_\ell)\ar[u] \ar[r]&\widetilde\H^2(X,\mu_\ell)\ar[u] &  & 
  } 
\end{split}
\end{equation}
which combined with Lemma \ref{g4} gives an element
of $\widetilde\H^2(X,\mu_\ell)$ whose image in $\widetilde\H^2(X_{k_v},\mu_\ell)$
does not come from $\H^2(\X,\mu_\ell)$. By diagram (\ref{11}) we get an element
of $\Br(X)[\ell]$ whose image in $\Br(X_{k_v})[\ell]$ is not contained in $
\Br(\X)[\ell]\oplus \Br({k_v})[\ell]$, as required.
\end{proof}

\begin{proof}[Proof of Theorem \ref{theo: main}] Since the question is of potential relevance to the Brauer--Manin obstruction, we are free to extend $k$, and thus assume that $X(k)\neq \emptyset$. 

Let $\ell\neq p$. Then there exists a smooth curve $C$ with function field $k$, and a smooth projective morphism $\pi\colon \X\to C$ with generic fibre $X$. By standard theorems in \'etale cohomology, each ${\bf R}^q\pi_*\mu_{\ell^n}$ is a finite \'etale group scheme over $C$, and $({\bf R}^1\pi_*\mu_\ell)_k={\bf Pic}_{X/k}[\ell]$.

Let $v\in |C|$. If ${\bf Pic}_{X/k}[\ell]=0$, then $v$ is not potentially relevant to the $\ell$-primary Brauer--Manin obstruction for $X$ by Proposition \ref{cor: UX0}. If ${\bf Pic}_{X/k}[\ell]\neq 0$, then $v$ is potentially relevant to the $\ell$-primary Brauer--Manin obstruction for $X$ by Theorem \ref{pot rel local}.

Now let $\ell=p$, and take $\pi\colon \X\to C$ as above. If ${\bf Pic}_{X/k}[p]=0$, then by \cite[Theorem 1.4]{BO} we can shrink $C$ to ensure that 
${\bf R}^1\pi_{*}\mu_{p}=0$, and ${\bf R}^q\pi_{*}\mu_{p}$ is an affine group scheme of finite type over $C$ for $q=2,3$. By Lemma \ref{lem: UX}, if $U_X=0$, we may further shrink $C$ to ensure that ${\bf R}^2\pi_{*}\mu_{p}$ is a finite flat group scheme over $C$. Hence no place $v\in |C|$ is potentially relevant for the $p$-primary Brauer--Manin obstruction for $X$ by Proposition \ref{cor: UX0}.

Finally, assume that either ${\bf Pic}_{X/k}[p]\neq 0$ or $U_X\neq 0$. Then again applying \cite[Theorem 1.4]{BO} we can shrink $C$ to ensure that 
${\bf R}^1\pi_{*}\mu_{p}$ is a finite flat group scheme, and ${\bf R}^2\pi_{*}\mu_{p}$ by an affine group scheme of finite type. Hence every place $v\in|C|$ 
is potentially relevant to the $p$-primary Brauer--Manin obstruction for $X$ by Theorem \ref{pot rel local}.
\end{proof}

\brem{\rm \label{2.7}
Let $X$ be a smooth, projective, geometrically integral variety over a global field $k$
of {\em any} characteristic $p\geq 0$.
If ${\bf Pic}_{X_{\bar k}/\bar{k}, {\rm red}}^\circ\neq 0$, then one can show that
every non-archimedean 
place of $k$ is potentially relevant to the Brauer--Manin obstruction without using
the results of \cite{KM}.

Indeed, up to replacing $k$ by a finite field extension, we can assume that 
${\bf Pic}_{X/k, {\rm red}}^\circ$ is a non-zero abelian variety over $k$, and that $X$ has a $k$-point  $Q$. We let $A=({\bf Pic}_{X/k, {\rm red}}^\circ)^\vee\neq 0$ be the Albanese variety of $X$, and ${\rm alb}\colon X\to A$ the Albanese morphism, uniquely determined by the condition ${\rm alb}(Q)=0$.
By the universal property of the Albanese variety, $A$ is the smallest abelian variety containing
${\rm alb}(X)$. In particular, ${\rm alb}(X)\neq \{0\}$. Thus, up to further replacing $k$ by 
a finite field extension, we can assume that $X$ contains a $k$-point $P$ such that
$R:={\rm alb}(P)\neq 0$. 

Let $w$ be a place of $k$. 
Let $Z:=\{0,R\}\subset A$ be the closed subscheme which is the union of the $k$-points $0$ and $R$.
Let $x\in Z({\bf A}_k)$ be such that $x_w=R$ and $x_v=0$ if $v\neq w$.
By \cite[Proposition 3.6]{Sto} (for $p=0$) and
\cite[Proposition 3.1]{CV26} (for $p>0$) we have $x\notin A({\bf A}_k)^\Br$, thus there exists
an element $\mathcal{B}\in \Br(A)$ such that $\sum_v\inv_v(\mathcal{B}(x))\neq 0\in\Q/\Z$.
Then $\mathcal{A}:=\mathcal{B}-\mathcal{B}(0)\in\Br(A)$ satisfies
$\mathcal{A}(0)=0$ and $\inv_w(\mathcal{A}(R))\neq 0$, hence
${\rm alb}^*(\mathcal{A})\in \Br(X)$ has non-constant evaluation at $k_w$-points of $X$.}
\erem

If $\ell=p$, then the smooth curve $C$ required in the proof of Theorem \ref{theo: main} seems rather difficult to determine explicitly, but in particular cases we can say more. We discuss these cases in the next two sections.

\section{Constant varieties}\label{2.5}

In this section, we consider \emph{constant} varieties over $k$. We will therefore keep the same notation as in the previous section, except that now $X$ will be defined over the constant field $\F$ of $k$. In this case, the sheaves ${\bf R}^q\pi_{*}\mu_{p^n}$ are obtained via base change along $C\to\Spec\, \F$, hence in the proof of Theorem \ref{theo: main} we may take $C$ to be the unique complete smooth curve with function field $k$. Thus:

\bthe \label{a1}
Let $X$ be a smooth, projective, geometrically integral variety over a finite field $\F$, and $\ell$ a prime. Assume that ${\bf Pic}_{X/\F}[\ell]=0$, and, if $\ell=p$, that $U_X=0$. Then for any global function field $k/\F$, any place $v$ of $k$, and any $\mathcal{A} \in \Br(X_{k_v})[\ell^\infty]$, the evaluation map
\[ {\rm ev}_{\mathcal{A}} \colon X(k_v)\to \Q_\ell/\Z_\ell \]
is constant.
\ethe

\bco \label{cor: const zero} Let $X$ be a smooth, projective and geometrically integral
variety over a finite field $\F$, with torsion free Picard scheme and $U_X=0$. Then, for any global function field $k/\F$, the Brauer--Manin pairing
\[ X({\bf A}_k)\times \Br(X_k) \to \Q/\Z \]
is identically zero.
\eco

\begin{proof} By the Lang--Weil estimates, $X$ admits a $0$-cycle of degree $1$, and hence so does $X_k$. Write this $0$-cycle as $\sum_i n_i Q_i$ and let $k_i=k(Q_i)$. Thus $\sum_i n_i[k_{i}:k]=1$. 

Let $(P_v)_v\in \prod_v X(k_v)$ and let $\mathcal{A}\in \Br(X_k)$. 
For any place $w$ of $k_i$ over $v$ we have
\[ {\rm inv}_w\,\mathcal{A}(Q_i)={\rm inv}_w({\rm res}_{k_{i,w}/k_v}\,\mathcal{A}(P_v))=[k_{i,w}:k_v]\cdot{\rm inv}_v\,\mathcal{A}(P_v),\]
where the first equality is from Theorem \ref{a1}, and the second is \cite[Proposition 1.4.7]{CTS21}. Global reciprocity now gives 
$\sum_w {\rm inv}_w\,\mathcal{A}(Q_i)=0$ for all $i$, and hence 
\begin{align*} 
0 &=\sum_i n_i\sum_{w} {\rm inv}_w\,\mathcal{A}(Q_i)=\sum_i n_i\sum_v\sum_{w\mid v} [k_{i,w}:k_v]\cdot{\rm inv}_v\,\mathcal{A}(P_v)  \\
&= \left(\sum_i n_i[k_i:k]\right)\left(\sum_v{\rm inv}_v\,\mathcal{A}(P_v)\right) = \sum_v{\rm inv}_v\,\mathcal{A}(P_v),
\end{align*}
as required.
\end{proof}

\brem This corollary can be proved more directly. By Lemma \ref{a3},
when $U_X=0$ and $K/\F$ is algebraically closed, 
the group ${\Br}(X_{K})$ is independent of $K$. When ${\bf Pic}_{X/\F}$ is torsion-free, one can use this, together with the Hochschild--Serre spectral sequence, to show that the natural map
\[ \Br(X)\oplus \Br(k) \to \Br(X_k) \]
is an isomorphism. Any element coming from $\Br(k)$ pairs trivially with $X({\bf A}_k)$. 
Using the valuative criterion of properness, together with the fact that $\Br(\O_v)=0$ for all $v$, 
one shows that the same is true for any element of $\Br(X)$.
\erem

In the case that $U_X\neq 0$ we can, in the proof of Theorem \ref{theo: main}, again take $C$ to be the unique complete smooth curve with function field $k$. We then obtain the following result.

\bthe \label{theo: const UXn0}
Let $X$ be a smooth, projective, geometrically integral
variety over a finite field $\F$. Let $\ell$ be a prime. Assume that ${\bf Pic}_{X/\F}[\ell]\neq 0$, or $\ell=p$ and $U_X\neq 0$. Then for any global function field $k/\F$,
every place of $k$ is potentially relevant to the $\ell$-primary Brauer--Manin obstruction for $X_k$. 
\ethe

\section{\texorpdfstring{$\bm{p}$}{p}-torsion free geometric Brauer group} \label{t-free}

For varieties over a global function field
with torsion free Picard scheme and $p$-torsion free geometric Brauer group,
we can give an explicit sufficient condition for a prime not to be potentially relevant to the Brauer--Manin obstruction. We say that $X$ has {\em good reduction at} $v$ if there exists a smooth projective scheme over $\O_v$ with geometrically integral fibres and generic fibre $X_{k_v}$.

\bpr \label{a4}
Let $X$ be a smooth, projective and geometrically integral variety over a global function field $k$
of characteristic $p$, with torsion-free Picard scheme, and such that $\Br(X_{\bar k})[p^\infty]=0$.
Any place of $k$ where $X$ has good reduction 
is not potentially relevant to the Brauer--Manin obstruction.
\epr
\begin{proof}
Since ${\bf Pic}_{X/k}$ is \'etale, the natural map
$\Br(X_{k^\s})\to\Br(X_{\bar k})$ is injective, see 
\cite[Theorem 5.2.5 (ii)]{CTS21}, \cite[Corollary 3.4]{D'A}, so we have
$\Br(X_{k^\s})[p^\infty]=0$. We conclude that
$\Br(X)=\Br_1(X)+\Br(X)(p')$, where $\Br_1(X)$ is the kernel of 
the natural map $\Br(X)\to\Br(X_{k^\s})$ and $\Br(X)(p')$ is the subgroup of elements
of order prime to $p$. 

Let $v$ be a place of $k$ where $X$ has good reduction. 

For any $\mathcal A\in\Br(X)(p')$ the evaluation map 
${\rm ev}_{\mathcal A}\colon X(k_v)\to\Q/\Z$ 
is constant by Proposition \ref{cor: UX0} applied with $\ell\neq p$,
whose conditions are satisfied in this case. 

The same holds for any $\mathcal A\in \Br_1(X)$ by \cite[Proposition 2.3]{CTS13} or 
\cite[Lemma 10.4.1, Proposition 10.4.2]{CTS21}. In {\em loc.~cit.}~these 
results are stated under the assumption ${\rm char}(k)=0$, 
but the same proof works in any characteristic.
We sketch it for the convenience of the reader. 
Let us prove that every element of 
$\Br_1(X_{k_v}):= {\rm ker}(\Br(X_{k_v})\to \Br(X_{k_v^\s}))$
has constant evaluation on $X(k_v)$.

Let $k_v^\nr\subset k_v^\s$ be the maximal unramified extension of $k_v$.
Let $I=\Gal(k_v^\s/k_v^\nr)\subset\Gal(k_v^\s/k_v)$ be the inertia subgroup. 
%, so that $\Gal(k_v^\s/k_v)/I=\Gal(k_v^\nr/k_v)=\Gal(\ov\F_v/\F_v)$.
Choose a prime $\ell\neq p$. The smooth base change theorem implies that the 
restriction of the natural action of $\Gal(k_v^\s/k_v)$ on $\H^2_\et(X_{k_v^\s},\Z_\ell(1))$
to $I$ is trivial. By assumption, $\Pic(X_{k_v^\s})\cong\Pic(X_{\ov k_v})$
is a finitely generated free abelian group, so the Kummer
exact sequence gives a Galois-equivariant embedding of $\Pic(X_{k_v^\s})$ into
$\H^2_\et(X_{k_v^\s},\Z_\ell(1))$, hence the natural action of $I$ on $\Pic(X_{k_v^\s})$
is trivial. This implies $\H^1(I,\Pic(X_{k_v^\s}))=0$. The Hochschild--Serre spectral sequence
$$\H^p(I,\H^q(X_{k_v^\s},\G_m))\Rightarrow \H^{p+q}(X_{k_v^\nr},\G_m)$$
gives the exact sequence
$$\Br(k_v^\nr)\to {\rm ker}(\Br(X_{k_v^\nr})\to \Br(X_{k_v^\s}))\to \H^1(I,\Pic(X_{k_v^\s}))=0.$$
Since $\Br(k_v^\nr)=0$ by \cite[Theorem 1.2.15(ii)]{CTS21}, we conclude that 
$\Br_1(X_{k_v})$ equals ${\rm ker}(\Br(X_{k_v})\to \Br(X_{k_v^\nr}))$.
By \cite[Lemma 10.4.1]{CTS21} (whose proof works {\em verbatim}), every element of
this group can be written as a sum of an element of $\Br(k_v)$ and an element of the Brauer
group of a smooth proper model of $X_{k_v}$ over $\O_v\subset k_v$.
This implies the constancy of the evaluation map.

These arguments can be used after any finite extension of $k$, thus proving the proposition.
\end{proof}

For surfaces, there is the following explicit criterion for the triviality of the geometric Brauer group.

\bpr \label{p1}
Let $X$ be a smooth, projective and geometrically integral surface over a field $k$ of characteristic $p$, such that ${\bf Pic}_{X/k}$ is smooth and $U_X=0$.
If $b_2=\rho+2h$ and $\NS(X_{\bar k})[p^\infty]=0$, then $\Br(X_{\bar k})[p^\infty]=0$.
\epr
\begin{proof}
By Lemma \ref{a2} the condition $b_2=\rho+2h$ implies that the divisible subgroup of
$\Br(X_{\bar k})[p^\infty]$ is trivial. Since $X$ is a surface, 
$\NS(X_{\bar k})[p^\infty]=0$ implies that $D_X=0$
by \cite[Theorem 1.2 (1)]{YY}. But $U_X=0$, so we have $\Br(X_{\bar k})[p^\infty]=0$. 
\end{proof}

Let us now discuss the case of K3 surfaces $X/k$. 
In this case, the Picard scheme ${\bf Pic}_{X/k}$ is torsion-free, and we have
$U_X=0$ if and only if $h>0$, which is equivalent
to  $X$ being non-supersingular. In this case $\Br(X_{\bar k})$ is divisible \cite[\S 7.2]{Ill}, so
to conclude that $\Br(X_{\bar k})[p^\infty]=0$ we only need to ensure that $\rho+2h=22$
and $h>0$.

\begin{exa} \label{exas}
\begin{enumerate}[nosep, label=(\roman*)]
\item Kazuhiro Ito proved that for $p\geq 5$ there exist K3 surfaces over $\ov\F_p$
with arbitrary positive integer values of $\rho$ and $h$ subject to the conditions that 
$\rho$ is even and $\rho+2h\leq 22$,
see \cite[Theorem 6.4]{Ito}. In particular, for any $h$ between 1 and 10
there are K3 surfaces over $\F_{p^n}$, for some $n\geq 1$, with $\rho=22-2h$.
\item %We note that ordinary elliptic curves exist over all finite fields.%AS is there a reference for this?
Let $E$ be an ordinary elliptic curve over $k=\F_p$, where $p$ is a prime
(including $p=2$). The Kummer surface $X=\Kum(E\times E)$ is defined as the minimal
desingularisation of the quotient of the abelian surface $A=E\times E$ by the involution 
$x\mapsto -x$. It is a K3 surface, so $b_2=22$. (This also holds when $p=2$ since 
$A$ is ordinary, see \cite{Kat}). We have
$\End(E_{\bar k})\cong\Z^{\oplus 2}$, hence the rank of $\NS(A_{\bar k})$ is 4.
The Picard rank of $X_{\bar k}$ is $\rho=16+4=20$ and $h_X=h_A=1$, see \cite[Remark 6.5]{LS23}.
\item \label{iii} Let $k$ be any field of positive characteristic $p$ (including $p=2$).
Let $E$ and $E'$ be elliptic curves over $k$ such that $E$ is ordinary and 
$E'$ is supersingular. Since 
$\Hom(E_{\bar k},E'_{\bar k})=0$, the rank of $\NS(A_{\bar k})$ is 2.
The Kummer surface $X=\Kum(E\times E')$ is defined in the same way as in (ii).
This is a K3 surface of Picard rank $\rho=16+2=18$ (again, see \cite{Kat} for the case $p=2$). We  have $h_X=h_A=2$ again by \cite[Remark 6.5]{LS23}.
\end{enumerate}
\end{exa}

Recall that a variety $X$ over a global function field $k$ is called {\em isotrivial} if there exist a finite extension $L/k$, a finite field $\F\subset L$, and a variety $X_0/\F$ such that $X_L \cong X_{0,L}$. 
Given the results of \S\ref{2.5}, Proposition \ref{a4} is only interesting in the non-isotrivial case.

If $k$ is a global function field, with fixed separable closure $k^{\rm s}$, we let $\overline{\F}$ denote the algebraic closure of the field of constants of $k$ inside $k^{\rm s}$, and define the {\em geometric Galois group} of $k$ to be $\Gal(k^\s/\overline{\F}k)$. If $X/k$ is isotrivial, then the action of the geometric Galois group of $k$ on any \'etale cohomology group of $X_{k^\s}$ has to be through a finite quotient. The following proposition shows that there exist non-isotrivial
varieties with torsion free Picard scheme and $p$-torsion free geometric Brauer group.

\bpr \label{a5}
In the situation of Example \ref{exas}\ref{iii}, assume that $E$ has non-constant $j$-invariant. Then $X$ is non-isotrivial. 
\epr
\begin{proof} We shall prove that, for any prime $\ell\neq p$, the geometric Galois group of $k$ does not act on $\H^2_\et(X_{k^{\rm s}},\Q_\ell)$ through a finite quotient. Indeed, since $E\times E'$ is non-supersingular, the standard computation of the cohomology groups of a Kummer surface shows that $\H^2_\et(X_{k^{\rm s}},\Q_\ell)$ contains $\H^2_\et((E\times E')_{k^{\rm s}},\Q_\ell)$, and thus $\H^1_\et(E_{k^{\rm s}},\Q_\ell)\otimes \H^1_\et(E'_{k^{\rm s}},\Q_\ell)$, as a direct summand. Since $E'$ is supersingular, it is isotrivial, and so the geometric Galois group of $k$ acts on $\H^1_\et(E'_{k^{\rm s}},\Q_\ell)$ through a finite quotient. After replacing $k$ by a finite separable extension, we may thus assume that the geometric Galois group acts trivially on $\H^1_\et(E'_{k^{\rm s}},\Q_\ell)$. It is therefore enough to show that the geometric Galois group of $k$ does \emph{not} act on $\H^1_\et(E_{k^{\rm s}},\Q_\ell)$ through a finite quotient. 

To see this, the fact that the $j$-invariant of $E$ is non-constant means there is a place $v$ of $k$ at which it is non-integral, and hence at which $E$ does not have potential good reduction. In particular, the N\'eron--Ogg--Shafarevich criterion shows that the inertia group at $v$ cannot act on $\H^1_\et(E_{k^{\rm s}},\Q_\ell)$ through a finite quotient. Since this inertia group is contained within the geometric Galois group of $k$, we are done.
\end{proof}

\section{The group scheme \texorpdfstring{${\bf R}^2\pi_*\mu_{p}$}{R2pi*mup}} \label{5}

In this section we give an example of a {\em supersingular}
K3 surface $X$ over a global function field $k$ of characteristic $p> 2$
such that $\Br(X)/\Br_0(X)$ is finite. This answers a question raised by the authors
in \cite[Introduction]{LS24}. Moreover, the finiteness of $\Br(X)/\Br_0(X)$ implies that
the Brauer--Manin set $X({\bf A}_k)^\Br$ is open in $X({\bf A}_k)$, in particular, 
only finitely many primes of $k$ are relevant to the Brauer--Manin
obstruction on $X$. This shows that in general one cannot drop `potentially' from
`potentially relevant' in the case $U_X\neq 0$ of Corollary~B.
This also shows that in general a field extension in \cite[Theorem 1.3]{V} is necessary.
The key feature of the example is that the connected component of ${\bf R}^2\pi_*\mu_p$
is a non-split form of $\G_a$, which for $p>2$ implies the finiteness of 
$\H^0(k,{\bf R}^2\pi_*\mu_p)$.

\ble \label{last}
Let $\pi\colon X\to\Spec(k)$ be a K3 surface over a finitely generated field $k$ of characteristic $p>0$.
Then $\Br(X)/\Br_0(X)$ is finite if and only if $\H^0(k,{\bf R}^2\pi_*\mu_p)$ is finite.
\ele
\begin{proof}
The finiteness of the prime-to-$p$ torsion subgroup 
$\big(\Br(X)/\Br_0(X)\big)(p')$ is established in \cite[Theorem 1.3]{SZ15} (for $p\neq 2$)
and in \cite{Ito18} (for $p=2$). The finiteness
of $\big(\Br(X)/\Br_0(X)\big)[p^\infty]$ is known when $X$ is not supersingular,
in which case ${\bf R}^2\pi_*\mu_p$
is a finite flat group scheme over $k$, see \cite[Theorem A]{LS24} or \cite{LQ24}.

We now consider $X$ supersingular. As usual, we define $\Br_1(X)$ to be the kernel of the map $\Br(X)\to\Br(X_{k^\s})$.
The natural map $\Br(X_{k_\s})\to\Br(X_{\bar k})$ is injective,
see \cite[Theorem 5.2.5 (i)]{CTS21} and \cite[Corollary 3.4]{D'A}, so
$\Br_1(X)$ is also the kernel of the map $\Br(X)\to\Br(X_{\bar k})$.
The Hochschild--Serre spectral sequence
$$\H^p(k,\H^q(X_{k^\s},\G_m))\Rightarrow\H^{p+q}(X,\G_m)$$
shows that $\Br_1(X)/\Br_0(X)$ is a subgroup of $\H^1(k,\Pic(X_{k^\s}))$, so it is finite
since $\Pic(X_{k^\s})$ is finitely generated and torsion-free. Since $X$ is a supersingular K3 surface, we have 
$p\Br(X_{\bar k})=0$, hence $\Br(X)/\Br_0(X)$ has finite exponent. 

The finiteness of $\big(\Br(X)/\Br_0(X)\big)[p^\infty]$ is thus equivalent to
the finiteness of $\big(\Br(X)/\Br_0(X)\big)[p]$.
Since $\Br(k)$ is $p$-divisible
\cite[Theorem 1.3.7]{CTS21}, the group $\Br_0(X)$ is $p$-divisible too. Thus we have
$$\big(\Br(X)/\Br_0(X)\big)[p]\cong \Br(X)[p]/\Br_0(X)[p].$$
The Kummer exact sequences for $k$ and $X$ give an exact sequence
$$\Pic(X)/p\to\widetilde\H^2(X,\mu_{p})\to\Br(X)[p]/\Br_0(X)[p]\to 0,$$
where $\widetilde\H^2(X,\mu_{p})$ is the cokernel of the natural map
$\H^2(k,\mu_{p})\to\H^2(X,\mu_{p})$.
Since $\Pic(X)/p$ is finite by the Lang--N\'eron theorem, the finiteness of
$\Br(X)[p]/\Br_0(X)[p]$ is equivalent to the finiteness of $\widetilde\H^2(X,\mu_{p})$.

Since $X$ is a K3 surface we have $\pi_*\mu_{p}=\mu_{p}$ and 
${\bf R}^1\pi_*\mu_{p}=0$. We also note that $\H^3(k,\mu_p)=0$, indeed if we let $\nu_p$ denote the cokernel of the map of \'etale sheaves $\G_m\overset{\times p}{\longrightarrow}\G_m$, then $\H^3(k,\mu_p)\cong \H^2_\et(k,\nu_p)=0$ since fields of characteristic $p$ have $p$-cohomological dimension $1$.  Using this, the Leray spectral sequence 
$\H^p(k,{\bf R}^q\pi_*\mu_{p})\Rightarrow \H^{p+q}(X,\mu_{p})$
gives an isomorphism $\widetilde\H^2(X,\mu_{p})\cong\H^0(k,{\bf R}^2\pi_*\mu_{p})$, and we may conclude.
\end{proof}

From now on we assume that $p\neq 2$. To explain our example, we need to recall some results of Ogus \cite{Ogu79, Ogu83} and of Bragg and Lieblich \cite{BL18} on the moduli space of supersingular K3 surfaces.

Recall that a {\em supersingular K3 lattice} is a lattice $\Lambda$ of rank 22
and signature $(1,21)$, such that
${\rm disc}(\Lambda\otimes\Q)=-1\in\Q^\times/\Q^{\times2}$ and $p(\Lambda^*/\Lambda)=0$.
The discriminant of $\Lambda$ is equal to $-p^{2\si}$, where $1\leq \si\leq 10$. 
The number $\si$ is called the {\em Artin invariant} of $\Lambda$. It is known that
$\si$ determines the isomorphism class of $\Lambda$. If $X$ is a K3 surface over an algebraically closed field, then a $\Lambda$-marking of $X$ is a homomorphism of lattices $\Lambda\to {\rm Pic}(X)$, where ${\rm Pic}(X)$ is endowed with the intersection pairing.

Let $S_{\Lambda}$ denote the moduli space over $\overline{\F}_p$ of 
$\Lambda$-marked supersingular K3 surfaces.
This is an algebraic space of finite type, locally separated and smooth of dimension $\si-1$, by 
\cite[Theorem 2.7]{Ogu83}. We write
$\varpi\colon Y\to S_\Lambda$ for the universal family of $\Lambda$-marked supersingular K3 surfaces.

Define $\Lambda_0=p\Lambda^*/p\Lambda$. This is an $\F_p$-vector space of dimension $2\si$
equipped with a non-degenerate symmetric bilinear form induced by the form on $\Lambda$.
This form is `non-neutral' in the sense that there are no
isotropic subspaces of dimension $\si$; such a bilinear form is unique up to isometry.

Let $k$ be a field of characteristic $p$. 
Let $F^*\colon \Lambda_0\otimes_{\F_p}k\to \Lambda_0\otimes_{\F_p}k$ be the 
Frobenius homomorphism sending $\lambda \otimes x$ to $\lambda\otimes x^p$. A subspace
$K\subset \Lambda_0\otimes_{\F_p}k$ is called {\em characteristic} if $K$ is isotropic, 
$\dim_k(K)=\si$ and $\dim_k(K+F^*K)=\si+1$. 
%The {\em Artin invariant} of a characteristic subspace $K$ is defined as $\si-\dim_{\F_p}(K\cap\Lambda_0)$.
The moduli space $M_{\Lambda_0}$ of characteristic subspaces of $\Lambda_0\otimes_{\F_p} \G_a$
is a smooth and projective variety over $\F_p$ of dimension $\si-1$, see \cite[\S4]{Ogu79}. 

A characteristic space $K$ is {\em strictly characteristic} if $K\cap\Lambda_0=0$. Let $U\subset
M_{\Lambda_0}$ be the dense open subset parameterising strictly characteristic subspaces.

Let $\mathscr{U}\to M_{\Lambda_0}$ denote the group scheme given by \cite[Definition 3.1.10]{BL18}. For a characteristic space $K$ over $k$, 
the fibre of $\mathscr U$ at $K$ is the kernel of the homomorphism 
\begin{equation}
(\Lambda_{0}\otimes_{\F_p}\G_a)/(K\otimes_k\G_a)\xrightarrow{1-F^*} 
(\Lambda_{0}\otimes_{\F_p}\G_a)/((K+F^*K)\otimes_k\G_a). \label{formula}
\end{equation}
There is a period morphism $\rho\colon S_\Lambda\to \overline{M}_{\Lambda_0}$,
where $\overline{M}_{\Lambda_0}:=M_{\Lambda_0}\times_{\F_p}\ov\F_p$, as defined in \cite[\S3]{Ogu83}. It follows from \cite[Propositions 3.1.11 and 3.5.9]{BL18} that 
$({\bf R}^2\varpi_*\mu_p)^\circ=\rho^*\mathscr{U}$.

The results of \cite{BL18} now give an upper bound on the Frobenius twist required to split $({\bf R}^2\pi_*\mu_p)^\circ$, for a general supersingular K3 surface. If $G$ is a group scheme in positive characteristic, we write $G^{(n)}$ for the $n$-fold Frobenius twist of $G$.

\bpr[Bragg--Lieblich] \label{prop: BL}
Let $k$ be a field of characteristic $p>2$, and let
$\pi\colon X\to \Spec\, k$ be a supersingular K3 surface of Artin invariant $\sigma$. 
Then $({\bf R}^2\pi_*\mu_p)^{\circ,(\sigma-1)}\cong \G_{a,k}$. 
\epr
\begin{proof}
There exists a finitely generated subfield $k_0\subset k$ over which $X$ is defined. Replacing $k$ by the compositum $k_0\overline{\F}_p$ we may assume that $k$ is finitely generated over $\overline{\F}_p$.
Passing to a finite separable extension of $k$ we may assume that ${\rm Pic}(X_{\bar k})={\rm Pic}(X)$.

Set $\Lambda:={\rm Pic}(X)$. By spreading out, we may assume that there exists an integral $\overline{\F}_p$-scheme $S$ of finite type with function field $k$, a relative K3 surface $\mathcal{X}\to S$ such that $\mathcal{X}_k=X$, and a marking $\Lambda_S\to {\bf Pic}_{\mathcal{X}/S}$ extending the isomorphism $\Lambda = {\rm Pic}(X)$. 
The marking of $\mathcal{X}$ gives rise to a morphism $S\to S_\Lambda$, so that
${\bf R}^2\pi_*\mu_p$ is the pullback of ${\bf R}^2\varpi_*\mu_p$ along the induced map $\Spec\,k \to S_\Lambda$.

Since $\Lambda={\rm Pic}(X)$, the composition with the period map $\rho$ gives rise to a map $\chi \colon \Spec\, k\to \overline{M}_{\Lambda_0}$, which lands in the open subspace 
$U\subset \overline{M}_{\Lambda_0}$of strictly characteristic subspaces. Since $({\bf R}^2\pi_*\mu_p)^\circ\cong \chi^*\mathscr{U}$, we now conclude using the fact that $\mathscr{U}_U^{(\sigma-1)}\cong \G_{a,U}$, as proved in \cite[Lemma 3.1.14]{BL18}.  
\end{proof}

Recall that the $a$-number of a supersingular abelian variety $A$ over $k$ is defined to be
$\dim_k\Hom(\alpha_{p},A[p])$. By passing to the Kummer surface of a supersingular abelian surface, we obtain the following corollary.

\bco \label{cococo}
Let $k$ be a field of characteristic $p>2$. Let $\pi_A\colon A\to \Spec\, k$ be a
supersingular abelian surface of $a$-number $a$.
Then $({\bf R}^2\pi_{A*}\mu_p)^{\circ,(2-a)}\cong \G_{a,k}$. 
\eco

\begin{proof} 
Let $\pi_Y \colon Y \to \Spec\,k$ be the blowup of $A$ along the closed subscheme
$A[2]$, which is finite \'etale over $k$. We denote by $\iota$
the involution on $Y$ that extends the involution $[-1]\colon A\to A$. 
The Kummer surface $X=\Kum(A)$ attached to $A$ is by definition the quotient 
$\pi_X\colon Y/\iota\to \Spec\,k$. 
The birational morphism $Y\to A$ and the double cover $Y\to X$ 
induce homomorphisms of finite type affine group schemes 
$${\bf R}^2\pi_{A*}\mu_p\to {\bf R}^2\pi_{Y*}\mu_p \leftarrow {\bf R}^2\pi_{X*}\mu_p.$$
These three group schemes are smooth. Indeed,
$A$ (respectively, $X$) is supersingular, thus 
$({\bf R}^2\pi_{A*}\mu_p)^\circ$ (respectively, $({\bf R}^2\pi_{X*}\mu_p)^\circ$)
is a form of $\G_{a,k}$ and so is smooth. The group scheme ${\bf R}^2\pi_{Y*}\mu_p$ is the
direct sum of ${\bf R}^2\pi_{A*}\mu_p$ and a finite \'etale group $k$-scheme $f_*(\Z/p)$,
where $f\colon A[2]\to\Spec\, k$ is the structure morphism, so it is also smooth.

We claim that the induced maps of connected components
$$({\bf R}^2\pi_{A*}\mu_p)^\circ\tilde\lra ({\bf R}^2\pi_{Y*}\mu_p)^\circ \tilde\longleftarrow 
({\bf R}^2\pi_{X*}\mu_p)^\circ$$
are isomorphisms. This is clear for $Y\to A$, so we consider $Y\to X$.
%Using \cite[Corollary 6.2.11]{CTS21} we see that the induced map on $\bar k$-points is injective with finite cokernel.
%Since $A$ is supersingular, ${\bf R}^2\pi_{A*}\mu_p$ is smooth. 
%We deduce that $({\bf R}^2\pi_{A*}\mu_p)^\circ \to ({\bf R}^2\pi_{Y*}\mu_p)^\circ$ is injective, with finite, connected cokernel. In particular, we get an isomorphism
%$({\bf R}^2\pi_{A*}\mu_p)^\circ\tilde\lra ({\bf R}^2\pi_{Y*}\mu_p)^\circ_{\rm red}$, and the latter is in fact a subgroup of  $({\bf R}^2\pi_{Y*}\mu_p)^{\circ}$.
The Kummer exact sequence gives a commutative diagram with exact rows
$$\xymatrix{0\ar[r]&\NS(Y_{\bar k})/p\ar[r]&\H^2(Y_{\bar k},\mu_p)\ar[r]&
\Br(Y_{\bar k})[p]\ar[r] & 0\\
0\ar[r]&\NS(X_{\bar k})/p\ar[r]\ar[u]&\H^2(X_{\bar k},\mu_p)\ar[r]\ar[u]&
\Br(X_{\bar k})[p]\ar[r]\ar[u]^\cong & 0}$$
By \cite[Proposition 4.1]{Sk} the right-hand vertical map is an isomorphism.
The left-hand vertical map is the reduction mod $p$ of the middle vertical arrow in the following commutative diagram with exact rows:
$$\xymatrix{0\ar[r]&\Z^{16}\ar[r]&\NS(Y_{\bar k})\ar[r]&
\NS(A_{\bar k})\ar[r] & 0\\
0\ar[r]&\Pi\ar[r]\ar[u]&\NS(X_{\bar k})\ar[r]\ar[u]&
\NS(A_{\bar k})\ar[r]\ar[u]^{\rm id} & 0}$$
where $\Pi$ is defined by the diagram. Counting ranks, we see that 
$\Pi$ is a free abelian group of rank 16.
Moreover, the 16 lines on $X_{\bar k}$ generate a subgroup $\Z^{16}\subset\Pi$ such that
the composition $\Z^{16}\to \Pi\to\Z^{16}$ is multiplication by $2$. Thus the reductions mod
$p$ of the vertical maps in the last diagram are isomorphisms. We deduce that
$\H^2(X_{\bar k},\mu_p)\to \H^2(Y_{\bar k},\mu_p)$ is an isomorphism.
This implies that the map of group schemes
${\bf R}^2\pi_{X*}\mu_p\to {\bf R}^2\pi_{Y*}\mu_p$ induces an isomorphism on
$\bar k$-points. Since both groups are smooth, the map must be an isomorphism,
so the claim is proved. It remains to observe that $X$ has Artin invariant $3-a$
and apply Proposition \ref{prop: BL}.
\end{proof}

For $\sigma=1$, Proposition \ref{prop: BL} is easy to check:
in this case we have $M_{\Lambda_0}\cong \Spec\, \F_{p^2}$ by \cite[Example 4.7]{Ogu79}, hence
(\ref{formula}) gives $({\bf R}^2\pi_*\mu_p)^\circ\cong\G_{a,k}$ for any supersingular K3 surface
over $k$ of Artin invariant 1.

Now consider the case $\sigma=2$. Then $M_{\Lambda_0}\cong \P^1_{\F_{p^2}}$ again by \cite[Example 4.7]{Ogu79}, and we can explicitly calculate the group scheme $\mathscr{U}$. In particular:

\bpr \label{5.3}
Let $p>2$. There exists a global function field $k$ of characteristic $p$
and a supersingular K3 surface $\pi\colon X\to\Spec\, k$ of Artin invariant $2$, 
such that the group $k$-scheme $({\bf R}^2\pi_*\mu_p)^\circ$ is a non-split form of $\G_{a}$, 
the group $\H^0(k,{\bf R}^2\pi_*\mu_p)$ is finite, and consequently $\Br(X)/\Br_0(X)$ is
finite. 
\epr

\begin{proof}
Let $\Lambda$ be a K3 lattice of Artin invariant $2$, and let $V=\Lambda_0$. Thus $V$ is a non-neutral quadratic space over $\F_p$ of dimension $4$, and thus the direct sum of a hyperbolic plane and the quadratic space corresponding to the form $Q(z,w)=z^2-aw^2$ for some $a\in \F_p^\times\setminus (\F_p^\times)^2$. 

We first explicitly calculate the generic fibre $\mathscr{U}_\eta$ of the group scheme $\mathscr{U}\to M_V$. To do so, we base change $V$ to $\F_{p^2}$, and choose $b\in \F_{p^2}$ such that $b^2=a$. Then setting $u=z-bw$, $v=z+bw$, we have
\[ Q_{V}(x,y,u,v)=2(xy+uv) \]
and the natural Frobenius $F^*$ on $V$ is given by $F^*(x,y,u,v)=(x^p,y^p,v^p,u^p)$. The space of isotropic vectors in $V$ is the quadric
\[ V(xy+uv)\subset \P^3 \]
which is isomorphic to $\P^1\times \P^1$, and the space of lines $\ell\subset V(xy+uv)$ is then just $\P^1\sqcup\P^1$, with the two families being given explicitly by
\begin{align*} u&=tx,\;\;v=-t^{-1}y,\\
u&=ty,\;\;v=-t^{-1}x, \end{align*}
where $t$ is a co-ordinate on $\P^1$. 

Hence one of the two generic characteristic subspaces $K\subset V\otimes\F_{p^2}(t)$ can be described as the linear span of the vectors $(1,0,t,0)$ and $(0,1,0,-t^{-1})$, in which case $K+F^*K$ is spanned by $K$ together with the vector $(0,0,-t,t^p)$. In particular, the kernel of 
\[ 1-F^*\colon  \frac{V\otimes\F_{p^2}(t)}{K}\to \frac{V\otimes\F_{p^2}(t)}{K+F^*K} \]
is the closed subscheme $\mathscr{U}_\eta\subset\G^2_{a}$ defined by the equation
\[ t^{p-1}(u-v^p)+v-u^p=0, \]
where $u,v$ are co-ordinates on $\G_{a}^2$. A linear change of co-ordinates transforms this into
\[ v^p=u-(t^{p^2}-t)^{p-1}u^p.\] 
Thus $\mathscr{U}_\eta$ is a non-split form of $\G_{a}$ since 
$-(t^{p^2}-t)^{p-1}$ is not a $p$th power in $\F_{p^2}(t)$, by 
\cite[Theorem 2.1, Corollary 2.3.1]{Rus}.

We now let $L$ denote the function field of an integral component of $S_\Lambda$, and $\varpi_L\colon Y_L\to\Spec\,L$ 
the pullback of the universal family to $L$. Since $\rho\colon S_\Lambda\to \overline{M}_V$ is \'etale, there exists a finite field extension $\iota\colon \overline{\F}_p(t)\to L$ such that 
$({\bf R}\varpi_{L*}\mu_p)^\circ\cong \iota^*\mathscr{U}_{\eta}$. 
Now pick a global function field $k\subset L$ such that $Y_L$ descends to a K3 surface $\pi\colon X\to \Spec\,k$, and $\iota$ to a morphism $\F_{q}(t)\to k$ for some $q$. Thus $({\bf R}\pi_*\mu_p)^\circ$ is non-split over the compositum $\overline{\F}_p k$, and hence non-split over $k$.

Since $({\bf R}^2\pi_*\mu_p)^\circ$ is non-split it
does not contain a copy of $\G_{a}$ and hence is {\em wound} in the sense of \cite[V, \S 5]{Oes}.
The condition $p>2$ then implies that $\H^0(k,({\bf R}^2\pi_*\mu_p)^\circ)$ is finite by 
\cite[VI, \S 3.1, Th\'eor\`eme, p.~65]{Oes}, thus
$\H^0(k,{\bf R}^2\pi_*\mu_p)$ is also finite. It remains to apply Lemma \ref{last}.
\end{proof}

\brem 
Up to a finite separable extension of $k$, the K3 surface $X$ constructed above is isomorphic to 
${\rm Kum}(A)$, where 
$\pi_A\colon A\to\Spec\,k$ is a supersingular abelian surface of $a$-number $1$, $({\bf R}^2\pi_{A*}\mu_p)^\circ$ is a non-split form of $\G_{a}$, and $\Br(A)/\Br_1(A)$ is finite.
Indeed, over an algebraically closed field, supersingular Kummer surfaces are exactly those of Artin invariant $\sigma\leq 2$ by \cite[Theorem 7.10]{Ogu79}. Since ${\bf Pic}_{X/k}$ is \'etale, we can 
replace $k$ by a finite separable extension in order to ensure that all $16$ smooth rational curves  
$\{C_i\}$ coming from the Kummer structure are $k$-rational. Since $\sum_i C_i\in 2{\rm Pic}(X)$ we then get a double cover $Y\to X$ ramified along the $C_i$. Blowing down the inverse images of the $C_i$  gives the required abelian surface over $k$.
\erem

\bigskip

\noindent{\Large {\bf Acknowledgements}}

\medskip

\noindent A.S.\ is grateful to Subhadip Majumder and Margherita Pagano
for helpful discussions. He thanks
the Lodha Mathematical Sciences Institute in Mumbai and the Max Planck
Institute for Mathematics in Bonn for excellent working conditions and support.
C.L.\ would like to thank the Max Planck Institute for Mathematics in Bonn for hospitality.

\bigskip

\noindent Department of Mathematics, Harrison Building, University of Exeter, EX4 4QF, United Kingdom

\medskip

\noindent \texttt{c.d.lazda@exeter.ac.uk}

\bigskip

\noindent Department of Mathematics, South Kensington Campus, Imperial College London,
SW7 2AZ, United Kingdom and Institute for the Information Transmission Problems, Russian Academy of Sciences, Moscow, 127994, Russia

\medskip

\noindent \texttt{a.skorobogatov@imperial.ac.uk}

\end{document}